 \documentclass[final,3p,times]{elsarticle}

\usepackage{amssymb, amsmath, amsthm}
\usepackage[]{natbib}

\newtheorem{theorem}{Theorem}
\newtheorem{definition}{Definition}
\newtheorem{lemma}{Lemma}

\journal{Journal of Math Analysis and Applications}

\begin{document}

\begin{frontmatter}


 %\title{\tnoteref{label1}}
%% \author{\corref{cor1}\fnref{label2}}
%% \ead{}
% \address{ \fnref{label3}}

\title{Mixed Estimates for Degenerate Multilinear Operators Associated to Simplexes}
 \author{Robert Kesler}

\address{Department of Mathematics, Cornell University, 212 Garden Avenue, Ithaca, NY 14853 }

\begin{abstract}
We prove that the degenerate trilinear operator $C_3^{-1,1,1}$ given by the formula

\begin{eqnarray*}
C_3^{-1,1,1}(f_1, f_2, f_3)(x)=\int_{x_1 < x_2 < x_3} \hat{f_1}(x_1) \hat{f_2}(x_2) \hat{f_3}(x_3) e^{2\pi i x (-x_1 + x_2 + x_3)} dx_1dx_2 dx_3
\end{eqnarray*}
satisfies the new estimates
\begin{eqnarray*}
 ||C_3^{-1,1,1}(f_1, f_2, f_3)||_{\frac{1}{\frac{1}{p_1}+\frac{1}{p_2}+\frac{1}{p_3}}} \lesssim_{p_1, p_2, p_3} ||\hat{f}_1||_{p^\prime_1}  ||f_2||_{p_2}||f_3||_{p_3}
 \end{eqnarray*}
for all  $f_1 \in L^{p_1}(\mathbb{R}): \hat{f}_1 \in L^{p_1^\prime}(\mathbb{R}) , f_2 \in L^{p_2}(\mathbb{R})$, and $f_3 \in L^{p_3}(\mathbb{R})$ such that $2 <p_1 \leq \infty, 1 < p_2, p_3 < \infty, \frac{1}{p_1}+\frac{1}{p_2} <1$, and $\frac{1}{p_2}+\frac{1}{p_3} <3/2$. Mixed estimates for some generalizations of $C_3^{-1,1,1}$ are also shown. \end{abstract}


\end{frontmatter}

 %\linenumbers


\section{Introduction}
\label{}

Many boundedness results have been obtained for singular multilinear integrals with nonclassical symbols, e.g. \citep{0238019,1491450,1999,2013,2420, 1887641,2127985,1957079,2199086}. A theorem due to Christ and Kiselev states that the bilinear operator $\tilde{C}^{\alpha_1, \alpha_2}_2$ initially defined on $L^1(\mathbb{R})$ functions by

\begin{eqnarray*}\label{1}
\tilde{C}^{\alpha_1, \alpha_2}_2(f_1, f_2)(x)= \int_{x_1<x_2}f_1(x_1) f_2(x_2) e^{2 \pi i x( \alpha_1 x_1 +\alpha_2 x_2)} dx_1 dx_2
\end{eqnarray*}
extends to a continuous map from $L^{p_1}(\mathbb{R}) \times L^{p_2}(\mathbb{R})$ into $L^{\frac{p_1^\prime p_2^\prime}{p_1^\prime + p_2^\prime}}(\mathbb{R})$, assuming $1\leq p_1, p_2 <2$ and $\alpha_1, \alpha_2 \not = 0$, see \citep{1809116,2013}.  Lacey and Thiele proved a wide range of $L^p$ estimates in \cite{1491450} for a related operator called the bilinear Hilbert transform  given by the formula
 
  \begin{eqnarray*}
BHT(f_1, f_2)(x)=  \tilde{C}_2^{1,1} (\hat{f}_1, \hat{f}_2)(x)=\int_{x_1 <x_2} \hat{f}_1(x_1) \hat{f}_2(x_2) e^{2\pi i x (x_1 +x_2)} dx_1 dx_2,
 \end{eqnarray*} 
after which boundedness was shown by Muscalu, Tao, and Thiele in \citep{2127985} for a trilinear variant of the BHT called the Biest, which takes the form
\begin{eqnarray*}
C^{1,1,1}_3(f_1, f_2, f_3)(x)=\int_{x_1<x_2<x_3} \hat{f}_1(x_1)\hat{f}_2(x_2) \hat{f}_3(x_3) e^{2\pi i x(x_1 +x_2+x_3)} dx_1 dx_2 dx_3.
 \end{eqnarray*}
However, multilinear integrals with sign degeneracies such as the operator
 \begin{eqnarray*}
 C_3^{-1,1,1}(f_1, f_2, f_3)(x)=  \int_{x_1 < x_2 <x_3} \hat{f}_1(x_1) \hat{f}_2(x_2) \hat{f}_3(x_3) e^{ 2\pi i x(-x_1+x_2+x_3)} dx_1dx_2 dx_3
 \end{eqnarray*} 
 are known to satisfy no $L^p$  estimates, see \citep{2013}. Despite this fact, we prove in Theorem \ref{OT} that there exists a constant $C_{p_1, p_2, p_3}$ such that for all  $f_1 \in L^{p_1}(\mathbb{R})$ satisfying  $\hat{f}_1 \in L^{p_1^\prime}(\mathbb{R})$ along with $f_2 \in L^{p_2}(\mathbb{R})$ and $f_3 \in L^{p_3}(\mathbb{R})$, 
\begin{eqnarray*}
 ||C_3^{-1,1,1}(f_1, f_2, f_3)||_{\frac{1}{\frac{1}{p_1}+\frac{1}{p_2}+\frac{1}{p_3}}} \leq C_{p_1, p_2, p_3} ||\hat{f}_1||_{p^\prime_1}  ||f_2||_{p_2}||f_3||_{p_3}
 \end{eqnarray*}
as long as $2 <p_1 \leq \infty, 1 < p_2, p_3 <\infty, \frac{1}{p_1}+\frac{1}{p_2} <1$ and $\frac{1}{p_2}+\frac{1}{p_3} <3/2$. We also establish mixed boundedness for $C_5^{1,1,-1,1,1}$ in Theorem \ref{IT} before handling $C_8^{1,1,-1,1,1-1,1,1}$ in Theorem \ref{SMT} and the main conclusion in this paper, namely Theorem \ref{MT}, which establishes mixed boundedness of the generalized n-linear degenerate integral

\begin{eqnarray*}
C_n^{\vec{\epsilon}}(f_1, ..., f_n)(x)= \int_{x_1<... <x_n} \hat{f}_1(x_1) ... \hat{f}_n(x_n) e^{2 \pi i x(\vec{\epsilon}\cdot \vec{x})} d\vec{x}, ~~~\vec{\epsilon} \in \{\pm 1\}
\end{eqnarray*}
for a large range of exponents and answers a previously open question posed by C. Muscalu. The proofs rely on the Christ-Kiselev martingale structure decomposition, see \citep{1809116}, in addition to a generalized version of the Littlewood-Paley inequality of Rubio de Francia for $L^p$ functions with $p<2$, see Rubio de Francia \cite{850681} and Lacey \cite{2293255}, and a maximal $l^2$ vector-valued inequality for the bilinear Hilbert transform and its generalizations, see Lemma \ref{L2}.

%NOTE: $A \lesssim B$ means there is a constant $C$, which may change from line to line, such that $A \leq CB$. The relations $A \gtrsim B$ and $A \simeq B$ are defined similarly. The space $L^p=L^p(\mathbb{R})$ unless otherwise stated.

 \section{Mixed Estimates}
\subsection{Preliminaries}
To introduce the martingale structure decomposition of Christ and Kiselev, we prove continuity for the map $\tilde{C}_2^{\alpha_1, \alpha_2}: L^{p_1}(\mathbb{R}) \times L^{p_2}(\mathbb{R}) \rightarrow L^{\frac{p_1^\prime p_2^\prime}{p_1^\prime + p_2^\prime}}(\mathbb{R})$ given by
\begin{eqnarray*}
\tilde{C}^{\alpha_1, \alpha_2}_2(f_1, f_2)(x)= \int_{x_1 < x_2} f_1 (x_1) f_2(x_2) e^{2 \pi i x (\alpha _1 x_1 +\alpha _2 x_2)} dx_1 dx_2
\end{eqnarray*}
for $p_1<2$, $p_2=2$ and $\alpha_1, \alpha _2 \not =0$ by following the argument in \citep{2013}. This is shown by first splitting the domain of integration $\{(x_1,x_2): x_1<x_2\}$ into disjoint sets depending on the weighted distance between $x_1$ and $x_2$. Specifically, define a map $\gamma_{f_2}:\mathbb{R} \rightarrow [0, 1]$ given by 

 \begin{eqnarray*}
 \gamma_{f_2}(x) = \frac{\int_{-\infty} ^x | f_2(\bar{x})|^{2} d\bar{x}}{||f_2||_2^2}
  \end{eqnarray*}
and form for every $m \in \mathbb{Z}^+ \cup \{0\}$ and $0 \leq j \leq2^j-2$ the martingale structure $E^{m}_{j}=\gamma_{f_2}^{-1}([j2^{-m}, (j+1)2^{-m}))$ and set $E^m_{2^m-1}=\gamma_{f_2}^{-1}([j2^{-m}, (j+1)2^{-m}])$. Then define $E^{m}_{j,l}=\gamma_{f_2}^{-1}([j2^{-m}, (j+1/2)2^{-m}))$ for $0 \leq j \leq 2^m-1$ and $E^{m}_{j,r}=\gamma_{f_2}^{-1}([(j+1/2)2^{-m}, (j+1)2^{-m}))$ for $0 \leq j \leq 2^m-2$ along with $E^{m}_{2^m-1,r}=\gamma_{f_2}^{-1}([1-2^{-m-1}, 1])$ to construct the partition

\begin{eqnarray}\label{Part}
\mathbb{R}^2 \supset \{ x_1 <x_2: \gamma_{f_2}(x_1)<\gamma_{f_2}(x_2)\} = \bigsqcup_{m \in \mathbb{Z}^+ \cup \{0\}} \bigsqcup_{0 \leq j <2^m} E^m_{j,l} \times E^m_{j,r}.
\end{eqnarray}
This decomposition separates points in $\{x_1 < x_2: \gamma_{f_2}(x_1) \not = \gamma_{f_2} (x_2)\}$ according to the smallest dyadic interval that contains both $\gamma_{f_2}(x_1)$ and $\gamma_{f_2}(x_2)$. Setting $S=\{ x_1 <x_2 : \gamma_{f_2}(x_1) = \gamma_{f_2}(x_2)\}$, it is immediate that $\int_{S} f_1(x_1) f_2(x_2) e^{ 2 \pi i x(x_1 +x_2)} dx_1 dx_2 =0.$ A quick computation then yields 

\begin{eqnarray*}
&& ||\tilde{C}_2(f_1, f_2) ||_{\frac{1}{1/2+1/p^\prime}} \\&=& \left| \left| \sum_{m \geq 0} \sum_{j=0}^{2^m-1} \tilde{C}_2(f_1 \chi_{E^m_{j,r} },f_2 \chi_{E^m_{j,l}} ) \right| \right| _{\frac{1}{1/2+1/p^\prime}}  ~\text{(using (\ref{Part}))} \\ &\leq& \sum_{m \geq 0} \sum_{j=0}^{2^m-1}  \left| \left| \widehat{ f_1 \chi_{E^m_{j, r}} } (\alpha _1 \cdot) \widehat{ f_2 \chi_{E^m_{j, l}}}(\alpha _2 \cdot)  \right| \right|_{\frac{1}{1/2+ 1/p^\prime}} \\ &\lesssim_{\vec{\alpha}}& \sum_{m \geq 0} \sum_{j=0}^{2^m-1} || f_1 \chi_{E^m_{j, r}} ||_{p} || f_2 \chi_{E^m_{j,l}} ||_2 ~(\text{by H\"older and Hausdorff-Young)} \\ &=& \sum_{m \geq 0}2^m  \left( \frac{1}{2^m} \sum_{j=0}^{2^m-1}|| f_1 \chi_{E^m_{j,l}}||^{p/p}_p \right)    | | f_2||_22^{-m/2}\\ &\leq& \sum_{m \geq 0} 2^{m/2} \left( \frac{1}{2^m} \sum_{j=0}^{2^m-1} ||f_1 \chi_{E^m_{j,l}}||_p ^p \right)^{1/p}  ||f_2||_2~\text{(by concavity)} \\ &\lesssim&\sum_{m \geq 0} 2^{m(1/2-1/p)} ||f_1||_2 ||f_2||_p ~\text{(by disjointness of $\{E^m_{j,l}\}_j$)} \\ &\lesssim_p & ||f_1||_p ||f_2||_2~(\text{using} ~p <2).
\end{eqnarray*}
%provided $\epsilon(||f_1||^2)$ is chosen small enough. Conclude by observing

%\begin{eqnarray*}
%|| A (f_1, f_2) ||_{\frac{1}{1/2 + 1/p^\prime}} = \lim_{\epsilon \rightarrow 0} ||W_1 (\tilde{f}_1(\epsilon), f_2) || _{\frac{1}{1/2 + 1/p^\prime}} \lesssim ||f_1||_2 ||f_2||_p.
%\end{eqnarray*}

In fact, we could have assumed by a standard limiting argument that $f_2 \in L^1(\mathbb{R})\cap L^2(\mathbb{R})$ and $f_2(x)\not =0~a.e.~x \in \mathbb{R}$, so that $\gamma_{f_2}$ would be strictly increasing. This would in turn force the sets $\{x_1 < x_2: \gamma_{f_2}(x_1) = \gamma_{f_2}(x_2)\}$ and $\gamma_{f_2}^{-1}(1)$ to be empty. Also note that the above proof adapted the martingale structure to the $L^2$ function. It is worth pointing out that one could just as well have adapted the martingale to the $L^p$ function, with a slightly modified proof. This second approach turns out to be the right one to generalize mixed estimates to more complicated operators. Before we illustrate this idea in Theorem \ref{OT}, we first record a few definitions and Theorem \ref{RdF}, based on observations of Rubio de Francia in \cite{850681} and Lacey in \cite{2293255}, along with Theorem \ref{BHT}, which states the boundedness of BHT.

\begin{definition} For $n \geq 1$ and $\vec{\epsilon} \in \{\pm 1\}^n$, 
\begin{eqnarray*}
C_n^{\vec{\epsilon}}(f_1, ..., f_n)(x)= \int_{x_1<... <x_n} \hat{f}_1(x_1) ... \hat{f}_n(x_n) e^{2 \pi i x(\vec{\epsilon}\cdot \vec{x})} d\vec{x}.
\end{eqnarray*}
\end{definition}

\begin{definition}
For $1 \leq p \leq \infty$, the Wiener space $W_p$ is given by

\begin{eqnarray*}
W_p(\mathbb{R})=\{ f \in L^p(\mathbb{R}): \hat{f} \in L^{p^\prime}(\mathbb{R})\}
\end{eqnarray*}
where $\hat{f}$ for generic $f \in L^p(\mathbb{R})$ is defined as a tempered distribution. Moreover, $W_p$ is given the structure of a normed vector space with $||f||_{W_p} = ||\hat{f}||_{L^{p^\prime}}$. 
\end{definition}
As sets,  $W_p \subset L^p$ is properly included for $p >2$, while $W_p=L^p$ for $p \leq 2$. 
\begin{theorem}[ \citep{2293255,850681}]\label{RdF}
Let $\{I_j\}_{j \in \mathbb{Z}}$ be a collection of disjoint rectangles in $\mathbb{R}^n$ for $n \geq 1$. Then the modified square function $\mathfrak{S}_r: L^p(\mathbb{R}^n) \rightarrow L^p(\mathbb{R}^n)$ given by

\begin{eqnarray*}
\mathfrak{S}_r(f)=\left( \sum_{j \in \mathbb{Z}} | f* \check{\chi}_{I_j}|^r \right)^{1/r}
\end{eqnarray*}
is continuous provided one of the following conditions holds: 

\begin{eqnarray*}
&&1)~2 \leq p <\infty~\text{and}~r=2 \\
&&2) ~1<p <2 ~\text{and}~ r >p^\prime. 
\end{eqnarray*}
\end{theorem}

\begin{theorem}[\cite{1491450,1999}]\label{BHT}
Let $1< p_1,p_2 \leq \infty$ satisfy $\frac{1}{p_1}+\frac{1}{p_2} <3/2$. Then the map $BHT: L^{p_1} \times L^{p_2} \rightarrow L^{\frac{1}{\frac{1}{p_1}+\frac{1}{p_2}}}$ is continuous, where 

\begin{eqnarray*}
BHT(f_1, f_2)(x)=\int_{x_1 <x_2} \hat{f}_1(x_1) \hat{f}_2(x_2) e^{ 2\pi i x(x_1+x_2)} dx_1 dx_2.
\end{eqnarray*}
\end{theorem}
\subsection{Statement and Proof of Mixed Estimates}

\begin{theorem}\label{OT}
The trilinear operator $C_3^{-1,1,1}: W_{p_1} \times L^{p_2} \times L^{p_3}\rightarrow L^{\frac{1}{\frac{1}{p_1}+\frac{1}{p_2}+\frac{1}{p_3}}}$ is bounded provided $2 < p_1 \leq \infty$, $1 < p_2, p_3 <\infty$, $\frac{1}{p_1} +\frac{1}{p_2} <1$, and $\frac{1}{p_2}+\frac{1}{p_3} < 3/2$. 
\end{theorem}

\begin{proof}
By a standard limiting argument, we assume $\hat{f}_1 \in L^1(\mathbb{R})\cap L^{p_1^\prime}(\mathbb{R})$, $\hat{f}_1(x) \not= 0$ for all $x \in \mathbb{R}$, and Fourier inversion holds. We introduce a martingale structure $E^m_j$ ala Christ and Kiselev given by

\begin{eqnarray*}
\gamma_{f_1}(x)&=&\frac{\int_{-\infty}^x |\hat{f}_1(\bar{x})|^{p_1^\prime} d\bar{x}}{||\hat{f}_1||_{{p_1^\prime}}^{p_1^\prime}}\\
E^m_j&=&\gamma_{f_1}^{-1}([j 2^{-m}, (j+1) 2^{-m}))
\end{eqnarray*}
so that $||\hat{f}_1 \chi_{E^{m}_j}||_{{p_1}^\prime} =2^{-m/p_1^\prime}$~for all $0 \leq j \leq  2^{m}-1$. As before, it is helpful to define
\begin{eqnarray*}
E^m_{j,l}&=&\gamma_{f_1}^{-1}([j 2^{-m}, (j+1/2)2^{-m}))\\
E^m_{j,r}&=&\gamma_{f_1}^{-1}([(j+1/2)2^{-m}, (j+1)2^{-m}))
\end{eqnarray*}
and construct the partition, using the fact that $\gamma_{f_1}$ is strictly increasing,

\begin{eqnarray*}
\mathbb{R}^2 \supset \{(x_1, x_2): x_1 < x_2\} = \bigsqcup_{m \in \mathbb{Z}^+ \cup \{0\}} \bigsqcup_{0 \leq j < 2^{m}} E^m_{j,l} \times E^m_{j,r}.
\end{eqnarray*} 
 We next split the proof into two cases, depending on whether the target exponent lies above or below 1. Our quasi-Banach analysis is not much different from the Banach version. 

CASE 1: q= $\frac{1}{\frac{1}{p_1}+\frac{1}{p_2}+\frac{1}{p_3}} \geq 1.$ Splitting the set $\{x_1 <x_2\}$ gives

\begin{eqnarray*}
&& ||C_3^{-1,1,1}(f_1, f_2, f_3)||_q \\ &=& \left |\left| \sum_{m \geq 0} \sum_{j=0}^{2^m-1}( f_1* \check{\chi}_{E^m_{j,l}} ) \cdot BHT(f_2*\check{\chi}_{E^m_{j,r}}, f_3) \right|\right|_q \\&\leq& \sum_{m \geq 0} \left| \left| \sum_{j=0}^{2^m-1}( f_1* \check{\chi}_{E^m_{j,l}} ) \cdot BHT(f_2*\check{\chi}_{E^m_{j,r}}, f_3) \right|\right|_q \\ &\leq& \sum_{m \geq 0} \left| \left| \left( \sum_{j=0}^{2^m-1}| f_1* \check{\chi}_{E^m_{j,l}}|^2 \right)^{1/2}\left( \sum_{j=0}^{2^m-1} | BHT(f_2*\check{\chi}_{E^m_{j,r}}, f_3)|^2 \right)^{1/2} \right|\right|_q.
\end{eqnarray*}
As before, the idea is produce a convergent geometric sum over the scale m. To obtain this, we first observe for $\frac{1}{q_1}=\frac{1}{p_2}+\frac{1}{p_3}$ that

\begin{eqnarray*}
&& \left| \left| \left( \sum_{j=0}^{2^m-1}| f_1* \check{\chi}_{E^m_{j,l}}|^2 \right)^{1/2}\left( \sum_{j=0}^{2^m-1} | BHT(f_2*\check{\chi}_{E^m_{j,r}}, f_3)|^2 \right)^{1/2} \right|\right|_q\\ &\leq& \left| \left| \left( \sum_{j=0}^{2^m-1}| f_1* \check{\chi}_{E^m_{j,l}}|^2 \right)^{1/2} \right| \right|_{p_1} \left| \left| \left( \sum_{j=0}^{2^m-1} | BHT(f_2*\check{\chi}_{E^m_{j,r}}, f_3)|^2 \right)^{1/2} \right|\right|_{q_1} \\ &\lesssim& \left| \left| \left( \sum_{j=0}^{2^m-1}| f_1* \check{\chi}_{E^m_{j,l}}|^2 \right)^{1/2} \right| \right|_{p_1} \left| \left| \left( \sum_{j=0}^{2^m-1} |f_2*\check{\chi}_{E^m_{j,r}}|^2 \right)^{1/2} \right| \right|_{p_2} ||f_3||_{p_3},
\end{eqnarray*}
where the last line follows from an application of Lemma \ref{L1} in the appendix and the known boundedness of the bilinear Hilbert transform. The advantage in writing the sum as a product in this way is that one may use H\"{o}lder's inequality even in the quasi-Banach case. 
Next, use convexity of $x \mapsto |x|^{{p_1}/2}$ and the Hausdorff-Young inequality to see

\begin{eqnarray*}
\left| \left| \left( \sum_{j=0}^{2^m-1}| f_1* \check{\chi}_{E^m_{j,l}}|^2 \right)^{1/2} \right| \right|_{p_1} & \leq& 2^{ m(1/2-1/p_1)} \left(\sum_{j=0}^{2^m-1}||f_1* \check{\chi}_{E^m_{j,l}}||_{p_1}^{p_1} \right)^{1/p_1} \\&\leq& 2^{ m(1/2-1/p_1)} \left(\sum_{j=0}^{2^m-1}||\hat{f}_1 \chi_{E^m_{j,l}}||_{p^\prime_1}^{p_1} \right)^{1/p_1} \\ &=& 2^{m(1/2-1/p_1)} 2^{m(1/p_1-1/p_1^\prime)} \\ &=& 2^{m(1/2-1/p_1^\prime)}.
\end{eqnarray*}
The remaining factor is $\left| \left| \left( \sum_{j=0}^{2^m-1} |f_2*\check{\chi}_{E^m_{j,r}}|^2 \right)^{1/2} \right| \right|_{p_2}$. If $p_2 \geq 2$, we may pass this problem to the original Rubio de Francia inequality in Theorem \ref{RdF} and conclude the theorem for CASE 1.   So, we may assume without loss of generality that $p_2<2$. For this, we need to invoke the generalized Rubio de Francia estimate by first raising the $l^2$ norm to the $l^r$ norm at an acceptable cost. Specifically, for any $r > p_2^\prime$, we compute

\begin{eqnarray*}
&& \left| \left| \left( \sum_{j=0}^{2^m-1} |f_2*\check{\chi}_{E^m_{j,r}}|^2 \right)^{1/2} \right| \right|_{p_2} \\&\leq& 2^{m(1/2-1/r)}\left| \left| \left( \sum_{j=0}^{2^m-1} |f_2*\check{\chi}_{E^m_{j,r}}|^r \right)^{1/r} \right| \right|_{p_2} \\ &\lesssim& 2^{m(1/2-1/r)} ||f_2||_{p_2}. 
\end{eqnarray*}
One checks that this loss does not affect the convergence of the sum over m because $1/2-1/p_1^\prime +1/2-1/r<0$ provided one chooses r close enough to $p_2^\prime$. 

CASE 2: $q<1$. Because one still has recourse to H\"{o}lder's inequality, the only difference with CASE 1 is how one moves the sum over m outside the $L^q$ norm in the absence of the triangle inequality by observing the following:

\begin{eqnarray*}
&& \left |\left| \sum_{m \geq 0} \sum_{j=0}^{2^m-1}( f_1* \check{\chi}_{E^m_{j,l}} ) \cdot BHT(f_2*\check{\chi}_{E^m_{j,r}}, f_3) \right|\right|_q \\ &\leq& \left( \sum_{m \geq 0} \left| \left| \sum_{j=0}^{2^m-1}( f_1* \check{\chi}_{E^m_{j,l}} ) \cdot BHT(f_2*\check{\chi}_{E^m_{j,r}}, f_3) \right|\right|_q^q \right)^{1/q}.
\end{eqnarray*}
\end{proof}

To prove our next result, Theorem \ref{IT},  we need the following fact from \cite{2221256}:
\begin{theorem}[Bi-Carleson Estimates]
The operator $\sup BHT (f_1, f_2)(x)$ given by 

\begin{eqnarray*}
\sup_{N} \left| \int_{x _1<x_2< N} \hat{f}_1(x_1) \hat{f}_2(x_2) e^{2 \pi i x(x_1+x_2)} dx_1 dx_2\right|
\end{eqnarray*}
 is bounded from $L^{p_1} \times L^{p_2}$ into $L^{\frac{p_1 p_2}{p_1+p_2}}$ if $1< p_1, p_2  \leq \infty$ and $\frac{1}{p_1}+\frac{1}{p_2} <3/2$.
\end{theorem}

\begin{theorem}\label{IT}
The operator $C_5^{1,1,-1,1,1}:L^{p_1} \times L^{p_2} \times W_{p_3} \times L^{p_4} \times L^{p_5}\rightarrow L^{\frac{1}{\sum_{i=1}^5 \frac{1}{p_i}}}$ is continuous provided $2 < p_3 \leq \infty$,  $1 < p_1, p_2, p_4, p_5 <\infty$, $\frac{1}{p_1}+\frac{1}{p_2}, \frac{1}{p_4}+\frac{1}{p_5} < 3/2$, and $\frac{1}{p_2}+\frac{1}{p_3}, \frac{1}{p_3}+\frac{1}{p_4} <1$.

\end{theorem}

\begin{proof}
One may try on a first attempt to introduce two copies of the same martingale structure, namely $E^{m_1}_{j_1}$ and $E^{m_2}_{j_2}$ adapted to $f_3$ this time, and split $C_5^{1,1,-1,1,1}$ as follows:

\begin{eqnarray*}
\sum_{m_1, m_2 \geq 0}\sum_{j_1, j_2} BHT(f_1, f_2 * \check{\chi}_{E^{m_1}_{j_1,l}}) (f_3* \check{\chi}_{E^{m_1}_{j_1, r}} * \check{\chi}_{E^{m_2}_{j_2,l}}) BHT(f_4*\check{\chi}_{E^{m_2}_{j_2,r}}, f_5).
\end{eqnarray*}
A computation similar to Theorem \ref{OT} yields in the Banach case, setting $\frac{1}{q_1}=\frac{1}{p_1}+\frac{1}{p_2}$ and $\frac{1}{q_2}=\frac{1}{p_4}+\frac{1}{p_5}$, 

\begin{eqnarray*}
&&||C_5^{1,1,-1,1,1} (\vec{f})||_q\\ &\lesssim&  \sum_{m_1, m_2 \geq 0} \left|\left| \left(\sum_{j_1} |BHT(f_1, f_2 * \check{\chi}_{E^{m_1}_{j_1, l}})|^2 \right)^{1/2} \right|\right|_{q_1} \times  \left|\left| \left( \sum_{j_1, j_2} |f_3* \check{\chi}_{E^{m_1}_{j_1, r}} * \check{\chi}_{E^{m_2}_{j_2, l}}|^2 \right)^{1/2} \right |\right|_{p_3}  \\ && ~~~~~~~\times\left|\left| \left( \sum_{j_2} |BHT(f_4*\check{\chi}_{E^{m_2}_{j_2,r}},f_5)|^2 \right)^{1/2} \right|\right|_{q_2} \\ &:=& \sum_{m_1, m_2 \geq 0} A_{m_1} \times B_{m_1, m_2} \times C_{m_2}.
\end{eqnarray*}
As before, the goal is to produce a convergent geometric series over the scales $m_1$ and $m_2$. The factors $A_{m_1}$ and $C_{m_2}$ are both handled by the $l^2$ vector-valued inequality for the $BHT$ using Theorem \ref{BHT} and Lemma \ref{L1} followed by the generalized Rubio de Francia estimate given in Theorem \ref{RdF}. If both $p_2, p_4 <2$, this part can be bounded above by 

\begin{eqnarray*}
2^{m_1(1/2-1/p_2^\prime)} 2^{m_2(1/2-1/p_4^\prime)} ||f_1||_{p_1} ||f_2||_{p_2} ||f_4||_{p_4} ||f_5||_{p_5}.
\end{eqnarray*}
 The decay that enables the geometric series to converge comes from the middle factor, namely $B_{m_1, m_2}$. Using the convexity of $x \mapsto |x|^{p_3/2}$, the Hausdorff-Young inequality, and the nesting of dyadic intervals, we can bound $B_{m_1, m_2}$ by

\begin{eqnarray*}
2^{\max\{m_1, m_2\} (1/2-1/p_3)}  \left(\sum_{j_1, j_2} ||\hat{f}_3\chi_{E^{m_1}_{j_1, r}}\chi_{E^{m_2}_{j_2, l}}||_{p^\prime_3}^{p_3}\right)^{1/p_3} = 2^{\max\{m_1, m_2\}(1/2-1/p_3^\prime)}.
\end{eqnarray*}
To finish in this case, we must require that the sum

\begin{eqnarray*}
\sum_{m_1, m_2 \geq 0}2^{m_1(1/2-1/p_2^\prime)} 2^{m_2(1/2-1/p_4^\prime)} 2^{\max\{m_1, m_2\}(1/2-1/p_3^\prime)}
\end{eqnarray*}
converges, which happens if and only if $\frac{1}{p_2}+\frac{1}{p_4}+\frac{1}{p_3} <3/2$. 
This condition implies $\frac{1}{p_2}+\frac{1}{p_3}, \frac{1}{p_3}+\frac{1}{p_4} <1$ by the assumption $p_2, p_4 <2$, so we have proven only a subset of the exponent range claimed in the theorem.  What cost us was the fact that 

\begin{eqnarray*}
B_{m_1, m_2}=2^{\max\{m_1, m_2\}(1/2-1/p_3^\prime)}
\end{eqnarray*}
did not decay fast enough for the sum over scales to converge. 

The key idea to get the full range is to adopt a different martingale structure decomposition that yields $B_{m_1, m_2}=2^{(m_1+m_2)(1/2-1/p_3^\prime)}$, which \emph{will} be enough to conclude the result. By another standard limiting argument, we assume $\hat{f}_3 \in L^1(\mathbb{R}) \cap L^{p_3^\prime}(\mathbb{R})$ and $\hat{f}_3 (x) \not = 0~a.e.~x \in \mathbb{R}$. First construct $E^{m_1}_{j_2}$ given by 

\begin{eqnarray*}
\gamma_{f_3}(x)&=&\frac{\int_{-\infty}^x |\hat{f}_3(\bar{x})|^{p_3^\prime} d\bar{x}}{||\hat{f}_3||_{p_3^\prime}^{p_3^\prime}}\\
E^{m_1}_{j_1}&=&\gamma_{f_3}^{-1}([j_12^{-m_1}, (j_1+1)2^{-m_1})).
\end{eqnarray*}
Next, define the \emph{restricted} martingale structure $E^{m_1, m_2}_{j_1, j_2}$ by setting

\begin{eqnarray*}
\gamma_{m_1, j_1, f_3}(x)&=&\frac{ \int_{-\infty}^x |\hat{f}_3(\bar{x})|^{p_3^\prime} \chi_{E^{m_1}_{j_1}} (\bar{x})d\bar{x}}{||\hat{f}_3 \chi_{E^{m_1}_{j_1}}||_{p_3^\prime}^{p_3^\prime}}\\
E^{m_1, m_2}_{j_1, j_2}&=&\gamma_{m_1, j_1, f_3}^{-1}([j_22^{-m_2}, (j_2+1)2^{-m_2}))~\forall~1 \leq j_2 \leq 2^{m_2}-2
\end{eqnarray*}
with the appropriate modification for $E^{m_1, m_2}_{j_1, 2^{m_2}-1}$. Thus, $||\hat{f}_3 \chi_{E^{m_1}_{j_1}} \chi_{E^{m_1, m_2}_{j_1, j_2}}||_{p_3^\prime} = 2^{-(m_1+m_2)/p_3^\prime}~\forall~0 \leq j_1 <2^{m_1},0 \leq j_2 < 2^{m_2}$. We now partition the domain 

\begin{eqnarray*}
\mathbb{R}^3 \supset \{ (x_2, x_3, x_4): x_2 <x_3 <x_4\}= \bigsqcup_{m_1, m_2} \bigsqcup_{j_1, j_2} E^{m_1}_{j_1, l} \times (E^{m_1}_{j_1, r} \cap E^{m_1, m_2}_{j_1, j_2, l} )\times E^{m_1, m_2}_{j_1, j_2, r}.
\end{eqnarray*}
CASE 1: $q=\frac{1}{\sum_{i=1}^{5} \frac{1}{p_i}} \leq 1$. Then 
\begin{eqnarray*}
\left| \left| C_{5}^{1,1,-1,1,1} (\vec{f}) \right| \right|_q \leq \left( \sum_{m_1, m_2 \geq 0} \left| \left| \sum_{j_1,j_2} BHT(f_1, ,f_2*\check{\chi}_{E^{m_1}_{j_1,l}}) \cdot f_{3}*\check{\chi}_{E^{m_1}_{j,r}} * \check{\chi}_{E^{m_1,m_2}_{j_1, j_2,l}} \cdot BHT(f_{4}*\check{\chi}_{E^{m_1,m_2}_{j_1,j_2,r}}, f_5) \right| \right|_q^q  \right)^{1/q}.
\end{eqnarray*}
We now want to split the sum over $j_1, j_2$ as
\begin{eqnarray*}
\sum_{j_1, j_2} =\sum_{j_1=0}^{2^{m_1}-1} \sum_{j_2 \not = 2^{m_2}-1} +\left.  \sum_{j_1=0}^{2^{m_1}-1} \right|_{j_2=2^{m_2}-1}
\end{eqnarray*}
to reflect the fact that $\{E^{m_1, m_2}_{j_1, j_2, r}\}_{j_1, j_2}$ is not a disjoint collection of intervals, while $\{E^{m_1, m_2}_{j_1, j_2,r}\}_{j_1, j_2: j_2 \not = 2^{m_2}-1}$ is. 

CASE 1a: We deal with the first term, which corresponds to $\sum_{j_1=0}^{2^{m_1}-1} \sum_{j_2 \not = 2^{m_2}-1}$. The computation is
\begin{eqnarray*}
&&\left( \sum_{m_1, m_2 \geq 0}  \left| \left|\sum_{j_1,j_2 \not = 2^{m_2}-1} BHT(f_1, ,f_2*\check{\chi}_{E^{m_1}_{j_1,l}}) \times  f_{3}*\check{\chi}_{E^{m_1}_{j,r}} * \check{\chi}_{E^{m_1,m_2}_{j_1, j_2,l}} \times BHT(f_{4}*\check{\chi}_{E^{m_1,m_2}_{j_1,j_2,r}}, f_5) \right| \right|_q^q  \right)^{1/q} \\ &\leq& \left( \sum_{m_1, m_2 \geq 0} \left| \left|  \sup_{j_1}| BHT(f_1, ,f_2*\check{\chi}_{E^{m_1}_{j_1,l}})| \times \left(\sum_{j_1,j_2 \not = 2^{m_2}-1} |f_{3}*\check{\chi}_{E^{m_1}_{j_1, r}} * \check{\chi}_{E^{m_1,m_2}_{j_1, j_2,l}}|^2 \right)^{1/2}  \right. \right. \right. \\ && \left. \left. \left. ~~~~~~~~~~ \times \left( \sum_{j_1, j_2 \not = 2^{m_2}-1} |BHT(f_{4}*\check{\chi}_{E^{m_1,m_2}_{j_1,j_2,r}}, f_5)|^2 \right)^{1/2} \right| \right|_q^q  \right)^{1/q}
\end{eqnarray*}
This in turn yields the upper bound 

\begin{eqnarray*}
&& \left( \sum_{m_1, m_2 \geq 0} \left| \left| \sup_{j_1} |BHT(f_1, f_2 * \check{\chi}_{E^{m_1}_{j_1}})| \right| \right|^q_{q_1} \times \left| \left| \left( \sum_{j_1}\sum_{j_2 \not = 2^{m_2}-1} |f_{3}*\check{\chi}_{E^{m_1}_{j_1,r}} * \check{\chi}_{E^{m_1, m_2}_{j_1, j_2,l}}|^2 \right)^{1/2} \right|\right|_{p_1}^q\right. \\ &&~~~~~~~~~~ \left. \times \left|\left| \left(\sum_{j_1}\sum_{j_2 \not = 0, 2^{m_2}-1}|BHT(f_{4}*\check{\chi}_{E^{m_1, m_2}_{j_1,j_2,r}}, f_5)|^2 \right)^{1/2} \right| \right|_{q_2}^q  \right)^{1/q} \\&:=& \left( \sum_{m_1, m_2 \geq 0} A_{m_1} \times B_{m_1, m_2} \times C_{m_1, m_2} \right)^{1/q}.
\end{eqnarray*}
To deal with $A_{m_1}$, we will use estimates for the Bi-Carleson operator. For $B_{m_1, m_2}$, we use the martingale structure to obtain $B_{m_1, m_2}<2^{q(m_1+m_2)(1/2-1/p_3^\prime)}$, and $C_{m_1, m_2}$ can be passed to the $l^2$ vector-valued story in Lemma \ref{L1} combined with the generalized Rubio de Francia estimate in Theorem \ref{RdF}. Following the same argument as before, the resulting geometric sum will be given for any $r > p_4^\prime$ by

\begin{eqnarray*}
\sum_{m_1, m_2 \geq 0} 2^{q(m_1+m_2)(1/2-1/p_3^\prime+1/2-1/r)}
\end{eqnarray*}
which converges provided one chooses r close enough to $p_4^\prime$ once one recalls that $\frac{1}{p_3}+\frac{1}{p_4}<1$ by assumption. 

CASE 1b: 
It only remains to tackle the endpoint case, i.e. the sum is over all $j_1$ for fixed $j_2 = 2^{m_2}-1$. Here, it is important to realize that the intervals $\{E^{m_1, m_2}_{j_1, 2^{m_2}-1, r}\}$ overlap.
The calculation begins with
\begin{eqnarray*}
&&\left( \sum_{m_1, m_2 \geq 0}  \left| \left|\sum_{j_1} BHT(f_1, ,f_2*\check{\chi}_{E^{m_1}_{j_1,l}}) \times  f_{3}*\check{\chi}_{E^{m_1}_{j,r}} * \check{\chi}_{E^{m_1,m_2}_{j_1, 2^{m_2}-1,l}} \times BHT(f_{4}*\check{\chi}_{E^{m_1,m_2}_{j_1,2^{m_2}-1,r}}, f_5) \right| \right|_q^q  \right)^{1/q} 
\end{eqnarray*}
and then uses suprema and Cauchy-Schwarz inequalities inside the $L^q$ norm followed by H\"{o}lder's inequality to yield
 \begin{eqnarray*}
&& \left( \sum_{m_1, m_2 \geq 0}  \left| \left| \left( \sum_{j_1} |BHT(f_1, ,f_2*\check{\chi}_{E^{m_1}_{j_1,l}})|^2 \right)^{1/2} \right. \right. \right. \\ && \left. \left. \left. \left( \sum_{j_1} |f_{3}*\check{\chi}_{E^{m_1}_{j,r}} * \check{\chi}_{E^{m_1,m_2}_{j_1, 2^{m_2}-1,l}}|^2 \right)^{1/2} \sup_{j_1} |BHT(f_{4}*\check{\chi}_{E^{m_1,m_2}_{j_1,2^{m_2}-1,r}}, f_5)| \right| \right|_q^q  \right)^{1/q} \\ &\leq&\left( \sum_{m_1, m_2 \geq 0}  \left| \left| \left( \sum_{j_1} |BHT(f_1, ,f_2*\check{\chi}_{E^{m_1}_{j_1,l}})|^2 \right)^{1/2} \right| \right|_{q_1}^q \times \right. \\ && \left. \left| \left|\left( \sum_{j_1} |f_{3}*\check{\chi}_{E^{m_1}_{j,r}} * \check{\chi}_{E^{m_1,m_2}_{j_1, 2^{m_2}-1,l}}|^2 \right)^{1/2} \right| \right|^q_{p_3} \left| \left| \sup_{j_1} |BHT(f_{4}*\check{\chi}_{E^{m_1,m_2}_{j_1,2^{m_2}-1,r}}, f_5)| \right| \right|_{q_2}^q  \right)^{1/q} \\ &=& \left(\sum_{m_1, m_2} A_{m_1} \times B_{m_1, m_2} \times C_{m_1, m_2} \right)^{1/q}.
\end{eqnarray*}
This time we pass $A_{m_1}$ to the $l^2$ vector-valued story followed by generalized Rubio de Francia, $B_{m_1, m_2} < 2^{qm_1(1/2-1/p_3^\prime)} 2^{-qm_2/p_3^\prime}$, and $C_{m_1, m_2}$ can be handled using the Bi-Carleson estimates. Therefore, the geometric sum one eventually  faces is of the form 

\begin{eqnarray*}
\sum_{m_1, m_2 \geq 0} 2^{qm_1(1/2-1/r)} 2^{qm_1 (1/2-1/p_3^\prime)} 2^{-qm_2 / p_3^\prime},
\end{eqnarray*}
which again converges for r close enough to $p_2^\prime$ because $\frac{1}{p_2}+\frac{1}{p_3} <1$.

CASE 2: $q>1$. One passes the sum over scales outside the $L^q$ norm using the triangle inequality before proceeding exactly as before. 

\end{proof}

\begin{theorem}\label{SMT}
The operator $C_8^{1,1,-1,1,1,-1, 1,1}:L^{p_1} \times L^{p_2} \times W_{p_3} \times L^{p_4} \times L^{p_5} \times W_{p_6} \times L^{p_7} \times L^{p_8}\rightarrow L^{\frac{1}{\sum_{i=1}^8 \frac{1}{p_i}}}$ is bounded provided $2<p_3, p_6 \leq \infty$, as well as $1 < p_1, p_2, p_4, p_5, p_7, p_8 <\infty$, $\frac{1}{p_1}+\frac{1}{p_2}, \frac{1}{p_4}+\frac{1}{p_5},\frac{1}{p_7}+\frac{1}{p_8} <3/2$, and $\frac{1}{p_2}+\frac{1}{p_3}, \frac{1}{p_3}+\frac{1}{p_4}, \frac{1}{p_5}+\frac{1}{p_6}, \frac{1}{p_6}+\frac{1}{p_7} <1$. 
\end{theorem}

\begin{proof}
The idea is two construct 4 martingale structures given by $E^{m_1}_{j_1}, E^{m_1, m_2}_{j_1, j_2},$ $E^{m_3}_{j_3},$ and $E^{m_3, m_4}_{j_3, j_4}$, where the first two are adapted to $f_3$, and the last two are adapted to $f_6$. Again, without loss of generality, $\hat{f}_3 \in L^1(\mathbb{R}) \cap L^{p_3^\prime}(\mathbb{R})$, $\hat{f}_6 \in L^1(\mathbb{R}) \cap L^{p_6^\prime}(\mathbb{R})$, $\hat{f}_3(x) , \hat{f}_6(x) \not =0~a.e. ~x \in \mathbb{R}$. Specifically, we define

\begin{eqnarray*}
\gamma_{f_3}(x)&=&\frac{\int_{-\infty}^x |\hat{f}_3(\bar{x})|^{p_3^\prime} d\bar{x} }{||\hat{f}_3||_{p_3^\prime}^{p_3^\prime}} \\ 
E^{m_1}_{j_1}&=&\gamma_{f_3}^{-1}([j_12^{-m_1}, (j_1+1)2^{-m_1}))\\ 
\gamma_{m_1, j_1, f_3}(x)&=&\frac{\int_{-\infty}^x |\hat{f}_3(\bar{x})|^{p_3^\prime} \chi_{E^{m_1}_{j_1}}(\bar{x}) d\bar{x}}{||\hat{f}_3 \chi_{E^{m_1}_{j_1}}||_{p_3^\prime}^{p_3^\prime}} \\ E^{m_1, m_2}_{j_1, j_2} &=&\gamma_{m_1, j_1, f_3}^{-1}([j_22^{-m_2}, (j_2+1)2^{-m_2}))\\
\gamma_{f_6}(x)&=&\frac{\int_{-\infty}^x |\hat{f}_6(\bar{x})|^{p_3^\prime} d\bar{x} }{||\hat{f}_6||_{p_6^\prime}^{p_6^\prime}} \\ 
\tilde{E}^{m_3}_{j_3}&=&\gamma_{f_6}^{-1}([j_32^{-m_3}, (j_3+1)2^{-m_3}))\\ 
\gamma_{m_3, j_3, f_6}(x)&=&\frac{\int_{-\infty}^x |\hat{f}_6(\bar{x})|^{p_6^\prime} \chi_{E^{m_3}_{j_3}}(\bar{x}) d\bar{x}}{||\hat{f}_6 \chi_{E^{m_3}_{j_3}}||_{p_6^\prime}^{p_6^\prime}} \\ \tilde{E}^{m_3, m_4}_{j_3, j_4} &=&\gamma_{m_3, j_3, f_6}^{-1}([j_22^{-m_2}, (j_2+1)2^{-m_2}))
\end{eqnarray*}
with the appropriate modifications for the rightmost elements of the restricted martingale structures. The hardest case is when $p_2, p_4, p_5, p_7 <2$, which places us in the quasi-Banach setting. We assume this now without loss of generality. Decomposing the operator $C^{1,1,-1,1,1,-1,1,1}$ yields

\begin{eqnarray*}
&&\left| \left| C_{5}^{1,1,-1,1,1} (\vec{f}) \right| \right|_q \\
&\leq& \left( \sum_{m_1, m_2, m_3, m_4 \geq 0} \left| \left| \sum_{j_1,j_2,j_3, j_4} BHT(f_1, ,f_2*\check{\chi}_{E^{m_1}_{j_1,l}}) \times f_{3}*\check{\chi}_{E^{m_1}_{j,r}} * \check{\chi}_{E^{m_1,m_2}_{j_1, j_2,l}} \right. \right. \right. \\ && \left. \left. \left.  ~~~~~~~~~~~~~~~~~~~~~~~~~~~~~~ \times BHT(f_{4}*\check{\chi}_{E^{m_1,m_2}_{j_1,j_2,r}}, f_5* \check{\chi}_{\tilde{E}^{m_3}_{j_3,l}}) \times (f_6*\check{\chi}_{\tilde{E}^{m_3}_{j_3,r}} * \check{\chi}_{\tilde{E}^{m_3, m_4}_{j_3, j_4, l}} ) \right. \right. \right. \\ && \left. \left. \left.  ~~~~~~~~~~~~~~~~~~~~~~~~~~~~~~\times BHT(f_7*\check{\chi}_{\tilde{E}^{m_3, m_4}_{j_3, j_4, r}},f_8) \right| \right|_q^q  \right)^{1/q}.
\end{eqnarray*}
We now separate the sum 

\begin{eqnarray*}
\sum_{j_1, j_2, j_3, j_4}&=&\sum_{j_1, j_3, j_2 \not = 2^{m_2}-1, j_4 \not = 2^{m_4}-1} + \left. \sum_{j_1, j_3, j_2 \not = 2^{m_2}-1} \right|_{j_4=2^{m_4}-1} \\&&+ \left. \sum_{j_1, j_3, j_4 \not = 2^{m_4}-1} \right|_{j_2=2^{m_2}-1} + \left. \sum_{j_1, j_3} \right|_{j_2=2^{m_2}-1, j_4=2^{m_4}-1} \\ &=& A + B +C +D,
\end{eqnarray*}
so that the corresponding estimate is broken into four pieces, $\tilde{A}, \tilde{B}, \tilde{C},$ and $\tilde{D}$.
The plan is now to bound each piece individually. To save space, it is helpful to define $\frac{1}{q_1}=\frac{1}{p_1}+\frac{1}{p_2}, \frac{1}{q_2}=\frac{1}{p_4}+\frac{1}{p_5}, \frac{1}{q_3}=\frac{1}{p_7}+\frac{1}{p_8}$.  First, 
\begin{eqnarray*}
\tilde{A}& :=& \left( \sum_{m_1, m_2, m_3, m_4 \geq 0} \left| \left| \sum_{j_1,j_3, j_2 \not =2^{m_2}-1 , j_4\not = 2^{m_4}-1} BHT(f_1, ,f_2*\check{\chi}_{E^{m_1}_{j_1,l}}) \times f_{3}*\check{\chi}_{E^{m_1}_{j,r}} * \check{\chi}_{E^{m_1,m_2}_{j_1, j_2,l}} \right. \right. \right. \\ && \left. \left. \left.  ~~~~~~~~~~~~~~~~~~~~~~~~~~~~~~~~~~~~~~~~~~~~~~~~ \times BHT(f_{4}*\check{\chi}_{E^{m_1,m_2}_{j_1,j_2,r}}, f_5* \check{\chi}_{\tilde{E}^{m_3}_{j_3,l}}) \right. \right. \right. \\ && \left. \left. \left.  ~~~~~~~~~~~~~~~~~~~~~~~~~~~~~~~~~~~~~~~~~~~~~~~~ \times (f_6*\check{\chi}_{\tilde{E}^{m_3}_{j_3,r}} * \check{\chi}_{\tilde{E}^{m_3, m_4}_{j_3, j_4, l}} ) \times BHT(f_7*\check{\chi}_{\tilde{E}^{m_3, m_4}_{j_3, j_4, r}},f_8) \right| \right|_q^q  \right)^{1/q}\\ &\leq &\left( \sum_{m_1, m_2, m_3, m_4 \geq 0} \left| \left| \sum_{j_3 , j_4\not = 2^{m_4}-1} \sup_{j_1} |BHT(f_1, ,f_2*\check{\chi}_{E^{m_1}_{j_1,l}})| \times\left( \sum_{j_1, j_2 \not = 2^{m_2}-1} | f_{3}*\check{\chi}_{E^{m_1}_{j,r}} * \check{\chi}_{E^{m_1,m_2}_{j_1, j_2,l}} |^2 \right)^{1/2}  \right. \right. \right. \\ && \left. \left. \left.   ~~~~~~~~~~~~~~~~~~~~~~~~~~~~~~~~ \times\left( \sum_{j_1, j_2 \not = 2^{m_2}-1} |BHT(f_{4}*\check{\chi}_{E^{m_1,m_2}_{j_1,j_2,r}}, f_5* \check{\chi}_{\tilde{E}^{m_3}_{j_3,l}})|^2 \right)^{1/2} \right. \right. \right. \\ && \left. \left. \left.~~~~~~~~~~~~~~~~~~~~~~~~~~~~~~~~ \times |f_6*\check{\chi}_{\tilde{E}^{m_3}_{j_3,r}} * \check{\chi}_{\tilde{E}^{m_3, m_4}_{j_3, j_4, l}} | \times |BHT(f_7*\check{\chi}_{\tilde{E}^{m_3, m_4}_{j_3, j_4, r}},f_8)| \right| \right|_q^q  \right)^{1/q}. \\ 
\end{eqnarray*}
Set $\vec{m} \geq 0$ to mean $m_i \geq 0$ for all components $i$. Then the previous calculation can be bounded above by 
\begin{eqnarray*}
&& \left( \sum_{\vec{m} \geq 0} \left| \left| \sup_{j_1} |BHT(f_1, ,f_2*\check{\chi}_{E^{m_1}_{j_1,l}})| \times  \left( \sum_{j_1, j_2 \not = 2^{m_2}-1} | f_{3}*\check{\chi}_{E^{m_1}_{j,r}} * \check{\chi}_{E^{m_1,m_2}_{j_1, j_2,l}} |^2 \right)^{1/2}  \right. \right. \right. \\ && \left. \left. \left.~~~~~~\times \sup_{j_3}\left( \sum_{j_1, j_2 \not = 2^{m_2}-1} |BHT(f_{4}*\check{\chi}_{E^{m_1,m_2}_{j_1,j_2,r}}, f_5* \check{\chi}_{\tilde{E}^{m_3}_{j_3,l}})|^2 \right)^{1/2}  \right. \right. \right. \\ && \left. \left. \left.~~~~~\times \left( \sum_{j_3, j_4 \not = 2^{m_4}-1} |f_6*\check{\chi}_{\tilde{E}^{m_3}_{j_3,r}} * \check{\chi}_{\tilde{E}^{m_3, m_4}_{j_3, j_4, l}} |^2 \right)^{1/2} \times \left( \sum_{j_3, j_4 \not = 2^{m_4}-1}  |BHT(f_7*\check{\chi}_{\tilde{E}^{m_3, m_4}_{j_3, j_4, r}},f_8)|^2 \right)^{1/2} \right| \right|_q^q  \right)^{1/q} .
\end{eqnarray*}
Using H\"{o}lder's inequality as before, we have an upper bound of the form 

\begin{eqnarray*}
&& \left( \sum_{\vec{m} \geq 0} \left| \left| \sup_{j_1} |BHT(f_1, ,f_2*\check{\chi}_{E^{m_1}_{j_1,l}})| \right| \right|_{q_1}^q    \left| \left|  \left( \sum_{j_1, j_2 \not = 2^{m_2}-1} | f_{3}*\check{\chi}_{E^{m_1}_{j,r}} * \check{\chi}_{E^{m_1,m_2}_{j_1, j_2,l}} |^2 \right)^{1/2} \right| \right|^q_{p_3} \right. \\ && \left. ~~~~~ \times \left| \left| \sup_{j_3}\left( \sum_{j_1, j_2 \not = 2^{m_2}-1} |BHT(f_{4}*\check{\chi}_{E^{m_1,m_2}_{j_1,j_2,r}}, f_5* \check{\chi}_{\tilde{E}^{m_3}_{j_3,l}})|^2 \right)^{1/2} \right| \right|_{q_2}^q \right. \\ &&  \left.  ~~~~~\times \left| \left| \left( \sum_{j_3, j_4 \not = 2^{m_4}-1} |f_6*\check{\chi}_{\tilde{E}^{m_3}_{j_3,r}} * \check{\chi}_{\tilde{E}^{m_3, m_4}_{j_3, j_4, l}} |^2 \right)^{1/2} \right| \right| _{p_6}^q  \times \left| \left| \left( \sum_{j_3, j_4 \not = 2^{m_4}-1}  |BHT(f_7*\check{\chi}_{\tilde{E}^{m_3, m_4}_{j_3, j_4, r}},f_8)|^2 \right)^{1/2} \right| \right|_{q_3}^q  \right)^{1/q} \\&:=& \left( \sum_{m_1, m_2, m_3, m_4 \geq 0} A_{m_1} B_{m_1, m_2} C_{m_1, m_2,m_3}, D_{j_3, j_4}, F_{m_3, m_4} \right)^{1/q}.
 \end{eqnarray*}
 The factor $A_{m_1}$ is handled by the Bi-Carleson operator estimates. For the other factors, $B_{m_1, m_2} < 2^{q(m_1+m_2)(1/2-1/p_3^\prime)}||\hat{f}_3||_{p_3^\prime}^q$, $C_{m_1, m_2, m_3}$ is handled using Lemma \ref{L2}, $D_{m_3, m_4}<2^{q(m_3+m_4)(1/2-1/p_6^\prime)}||\hat{f}_6||_{p_6^\prime}^{q}$, and $F_{m_3, m_4}$ is is handled using Lemma \ref{L1}. 
 Since we are assuming $p_2, p_4, p_5, p_7<2$, the geometric sum one eventually faces takes the form 
 
 \begin{eqnarray*}
 \sum_{m_1, m_2, m_3, m_4} 2^{q(m_1 + m_2)(1/2-1/p_3^\prime+1/2-1/p_4^\prime)} 2^{q(m_3+m_4)(1/2-1/p_6^\prime+1/2-1/p_7^\prime)},
 \end{eqnarray*}
which converges because $\frac{1}{p_3}+\frac{1}{p_4} , \frac{1}{p_6}+\frac{1}{p_7} <1$. The next term we face is
\begin{eqnarray*}
&& \tilde{B} := \left( \sum_{\vec{m} \geq 0} \left| \left|\left. \sum_{j_1,j_3, j_2 \not =2^{m_2}-1 ,} \right|_{j_4 = 2^{m_4}-1} BHT(f_1, ,f_2*\check{\chi}_{E^{m_1}_{j_1,l}}) \right. \right. \right. \\ && \left. \left. \left. ~~~~~~~~~~~~~~~~~~~~~~~~~~~~~~~~~~~~~~~~~~~~ \times (f_{3}*\check{\chi}_{E^{m_1}_{j,r}} * \check{\chi}_{E^{m_1,m_2}_{j_1, j_2,l}} )\times BHT(f_{4}*\check{\chi}_{E^{m_1,m_2}_{j_1,j_2,r}}, f_5* \check{\chi}_{\tilde{E}^{m_3}_{j_3,l}})  \right. \right. \right. \\ && \left. \left. \left.  ~~~~~~~~~~~~~~~~~~~~~~~~~~~~~~~~~~~~~~~~~~~~\times (f_6*\check{\chi}_{\tilde{E}^{m_3}_{j_3,r}} * \check{\chi}_{\tilde{E}^{m_3, m_4}_{j_3, j_4, l}} ) \times BHT(f_7*\check{\chi}_{\tilde{E}^{m_3, m_4}_{j_3, j_4, r}},f_8) \right| \right|_q^q  \right)^{1/q}. 
\end{eqnarray*}
It is readily seen that one has an upper bound of the form
\begin{eqnarray*}
\tilde{B} &\leq& \left( \sum_{\vec{m} \geq 0} \left| \left| \sup_{j_1} |BHT(f_1, ,f_2*\check{\chi}_{E^{m_1}_{j_1,l}})| \times   \left( \sum_{j_1, j_2 \not = 2^{m_2}-1} |f_{3}*\check{\chi}_{E^{m_1}_{j,r}} * \check{\chi}_{E^{m_1,m_2}_{j_1, j_2,l}}|^2 \right)^{1/2}\right. \right. \right. \\ && \left. \left. \left.  ~~~~~\times \left( \sum_{j_1, j_3, j_2 \not = 2^{m_2}-1} |BHT(f_{4}*\check{\chi}_{E^{m_1,m_2}_{j_1,j_2,r}}, f_5* \check{\chi}_{\tilde{E}^{m_3}_{j_3,l}})|^2 \right)^{1/2} \right. \right. \right. \\ && \left. \left. \left.  ~~~~~\times \left( \sum_{j_3} |f_6*\check{\chi}_{\tilde{E}^{m_3}_{j_3,r}} * \check{\chi}_{\tilde{E}^{m_3, m_4}_{j_3, 2^{m_4}-1, l}}|^2 \right)^{1/2} \times  \sup_{j_3} |BHT(f_7*\check{\chi}_{\tilde{E}^{m_3, m_4}_{j_3, 2^{m_4}-1, r}},f_8)| \right| \right|_q^q  \right)^{1/q},
\end{eqnarray*}
which is passed to H\"{o}lder's inequality as before, giving the expression

\begin{eqnarray*}
&& \left( \sum_{\vec{m} \geq 0} \left| \left| \sup_{j_1} |BHT(f_1, ,f_2*\check{\chi}_{E^{m_1}_{j_1,l}})| \right| \right|_{q_1}^q \times   \left| \left| \left( \sum_{j_1, j_2 \not = 2^{m_2}-1} |f_{3}*\check{\chi}_{E^{m_1}_{j,r}} * \check{\chi}_{E^{m_1,m_2}_{j_1, j_2,l}}|^2 \right)^{1/2} \right| \right|_{p_3}^q  \right. \\ && \left.~~~~~ \times \left| \left| \left( \sum_{j_1, j_3, j_2 \not = 2^{m_2}-1} |BHT(f_{4}*\check{\chi}_{E^{m_1,m_2}_{j_1,j_2,r}}, f_5* \check{\chi}_{\tilde{E}^{m_3}_{j_3,l}})|^2 \right)^{1/2} \right| \right|_{q_2}^q   \right. \\ && \left.~~~~~ \times\left| \left|  \left( \sum_{j_3} |f_6*\check{\chi}_{\tilde{E}^{m_3}_{j_3,r}} * \check{\chi}_{\tilde{E}^{m_3, m_4}_{j_3, 2^{m_4}-1, l}}|^2 \right)^{1/2} \right| \right|_{p_6}^q \times \left| \left| \sup_{j_3} |BHT(f_7*\check{\chi}_{\tilde{E}^{m_3, m_4}_{j_3, 2^{m_4}-1, r}},f_8)| \right| \right|_{q_3}^q  \right)^{1/q} \\ &:=& \left( \sum_{m_1, m_2, m_3, m_4 \geq 0} A_{m_1} \times B_{m_1, m_2}\times  C_{m_1, m_2, m_3} \times D_{m_3, m_4} \times F_{m_3, m_4} \right)^{1/q}.
\end{eqnarray*}
We pass $A_{m_1}$ and $F_{m_3, m_4}$ to the Bi-Carleson estimates, use the standard decay for $B_{m_1, m_2}$, use Lemma \ref{L1} and generalized Rubio de Francia for $C_{m_1, m_2, m_3}$, and observe $D_{m_3, m_4}<2^{qm_3(1/2-1/p_6^\prime)} 2^{-qm_4/p_6^\prime}$. Doing this, gives the geometric series

\begin{eqnarray*}
\sum_{m_1, m_2, m_3, m_4 \geq 0} 2^{q(m_1+m_2)(1/2-1/p_3^\prime + 1/2-1/p_4^\prime)} 2^{qm_3(1/2-1/p_6^\prime + 1/2-1/p_5^\prime)} 2^{-qm_4/p_6^\prime},
\end{eqnarray*}
which converges because $\frac{1}{p_3}+\frac{1}{p_4}, \frac{1}{p_5}+\frac{1}{p_6} <1$. 

The analysis for $\tilde{C}$ is the same as the analysis for $\tilde{B}$ except that the roles of $j_1, j_2$ and $j_3, j_4$ are reversed. 
Thus, it only remains to bound $\tilde{D}$ to obtain the result. For this, we observe

\begin{eqnarray*}
&& \tilde{D} := \left( \sum_{\vec{m} \geq 0} \left| \left|\left. \sum_{j_1,j_3} \right|_{j_2 = 2^{m_2}-1, j_4 = 2^{m_4}-1} BHT(f_1, ,f_2*\check{\chi}_{E^{m_1}_{j_1,l}})  \right. \right. \right. \\ && \left. \left. \left. ~~~~~~~~~~~~~~~~~~~~~~~~~~~~~~~~~~~~~~~~~~~  \times(f_{3}*\check{\chi}_{E^{m_1}_{j,r}} * \check{\chi}_{E^{m_1,m_2}_{j_1, j_2,l}} )\times BHT(f_{4}*\check{\chi}_{E^{m_1,m_2}_{j_1,j_2,r}}, f_5* \check{\chi}_{\tilde{E}^{m_3}_{j_3,l}}) \right. \right. \right. \\ && \left. \left. \left. ~~~~~~~~~~~~~~~~~~~~~~~~~~~~~~~~~~~~~~~~~~~ \times (f_6*\check{\chi}_{\tilde{E}^{m_3}_{j_3,r}} * \check{\chi}_{\tilde{E}^{m_3, m_4}_{j_3, j_4, l}} ) \times BHT(f_7*\check{\chi}_{\tilde{E}^{m_3, m_4}_{j_3, j_4, r}},f_8) \right| \right|_q^q  \right)^{1/q} \\ &\leq&\left( \sum_{\vec{m} \geq 0} \left| \left| \sum_{j_1} |BHT(f_1, ,f_2*\check{\chi}_{E^{m_1}_{j_1,l}})| \times  |f_{3}*\check{\chi}_{E^{m_1}_{j,r}} * \check{\chi}_{E^{m_1,m_2}_{j_1, 2^{m_2}-1,l}} |\right. \right. \right. \\ && \left. \left.  \left.~~~~~\times \left( \sum_{j_3} |BHT(f_{4}*\check{\chi}_{E^{m_1,m_2}_{j_1,2^{m_2}-1,r}}, f_5* \check{\chi}_{\tilde{E}^{m_3}_{j_3,l}})|^2 \right)^{1/2}   \right. \right. \right. \\ && \left. \left. \left. ~~~~~ \times\left( \sum_{j_3} |f_6*\check{\chi}_{\tilde{E}^{m_3}_{j_3,r}} * \check{\chi}_{\tilde{E}^{m_3, m_4}_{j_3, 2^{m_4}-1, l}} |^2 \right)^{1/2} \times \sup_{j_3} |BHT(f_7*\check{\chi}_{\tilde{E}^{m_3, m_4}_{j_3, 2^{m_4}-1, r}},f_8)| \right| \right|_q^q  \right)^{1/q} \\&\leq& \left( \sum_{\vec{m} \geq 0} \left| \left| \left( \sum_{j_1} |BHT(f_1, ,f_2*\check{\chi}_{E^{m_1}_{j_1,l}})|^2 \right)^{1/2} \times \left(  \sum_{j_1} |f_{3}*\check{\chi}_{E^{m_1}_{j,r}} * \check{\chi}_{E^{m_1,m_2}_{j_1, 2^{m_2}-1,l}} |^2 \right)^{1/2} \right. \right. \right. \\ && \left. \left.  \left. ~~~~~\times \sup_{j_1} \left( \sum_{j_3} |BHT(f_{4}*\check{\chi}_{E^{m_1,m_2}_{j_1,2^{m_2}-1,r}}, f_5* \check{\chi}_{\tilde{E}^{m_3}_{j_3,l}})|^2 \right)^{1/2}  \right. \right. \right. \\ && \left. \left. \left. ~~~~ \times \left( \sum_{j_3} |f_6*\check{\chi}_{\tilde{E}^{m_3}_{j_3,r}} * \check{\chi}_{\tilde{E}^{m_3, m_4}_{j_3, 2^{m_4}-1, l}} |^2 \right)^{1/2} \times \sup_{j_3} |BHT(f_7*\check{\chi}_{\tilde{E}^{m_3, m_4}_{j_3, 2^{m_4}-1, r}},f_8)| \right| \right|_q^q  \right)^{1/q}. 
\end{eqnarray*}

Using H\"{o}lder's inequality once more, we obtain the upper bound

\begin{eqnarray*}
 &&\left( \sum_{\vec{m} \geq 0} \left| \left| \left( \sum_{j_1} |BHT(f_1, ,f_2*\check{\chi}_{E^{m_1}_{j_1,l}})|^2 \right)^{1/2} \right| \right|_{q_1}^q \times  \left| \left| \left(  \sum_{j_1} |f_{3}*\check{\chi}_{E^{m_1}_{j,r}} * \check{\chi}_{E^{m_1,m_2}_{j_1, 2^{m_2}-1,l}} |^2 \right)^{1/2} \right| \right|_{p_3}^q \right. \\ &&\left.~~~~~~~~ \times  \left| \left| \sup_{j_1} \left( \sum_{j_3} |BHT(f_{4}*\check{\chi}_{E^{m_1,m_2}_{j_1,2^{m_2}-1,r}}, f_5* \check{\chi}_{\tilde{E}^{m_3}_{j_3,l}}|^2 \right)^{1/2}  \right| \right|_{q_2}^q  \right. \\ &&  \left.~~~~~~~~\times  \left| \left| \left( \sum_{j_3} |f_6*\check{\chi}_{\tilde{E}^{m_3}_{j_3,r}} * \check{\chi}_{\tilde{E}^{m_3, m_4}_{j_3, 2^{m_4}-1, l}} |^2 \right)^{1/2} \right| \right|_{p_6}^q \right. \\ &&\left.~~~~~~~~ \times \left| \left| \sup_{j_3} |BHT(f_7*\check{\chi}_{\tilde{E}^{m_3, m_4}_{j_3, 2^{m_4}-1, r}},f_8)| \right| \right|_{q_3}^q  \right)^{1/q}  \\ &:=& \left( \sum_{m_1,m_2, m_3, m_4} A_{m_1} B_{m_1, m_2} C_{m_1, m_2, m_3} D_{m_3, m_4} F_{m_3, m_4} \right)^{1/q}.
\end{eqnarray*}
One handles $A_{m_1}$ using the $l^2$ vector-valued for the BHT and generalized Rubio de Francia, $B_{m_1, m_2} < 2^{qm_1(1/2-1/p_3^\prime)} 2^{-qm_2/p_3^\prime}$, $C_{m_1, m_2, m_3}$ using Lemma \ref{L2}, $D_{m_3, m_4} < 2^{q m_3(1/2-1/p_6^\prime)} 2^{-q m_4 /p_6^\prime}$, and $F_{m_3, m_4}$ using Bi-Carleson estimates. 

The geometric series one eventually faces in this case is 

\begin{eqnarray*}
\sum_{m_1, m_2, m_3,m_4 \geq 0} 2^{q m_1 ( 1/2-1/p_2^\prime +1/2 - 1/p_3^\prime)} 2^{-q m_2 /p_3^\prime} 2^{qm_3(1/2-1/p_5^\prime +1/2-1/p_6^\prime)} 2^{-qm_4/p_6^\prime},
\end{eqnarray*}
which converges because $\frac{1}{p_2}+\frac{1}{p_3}, \frac{1}{p_5}+\frac{1}{p_6} <1$.

\end{proof}

It is important to note that each BHT in the previous proof could have been replaced by $C_n^{1,...,1}$ provided we had estimates for the maximal variant 

\begin{eqnarray*}
 \sup C_n^{1,...,1}(\vec{f})(x):= \sup_{M,N} \left| \int_{M<x_1 < ... < x_n <N} \hat{f}_1(x_1) ... \hat{f}_n(x_n) e^{2\pi i x (x_1 + ... + x_n)} d\vec{x} \right|.
\end{eqnarray*}
While such results have not yet appeared in published form, we shall assume them for the purposes of this paper based on personal communication with C. Muscalu.  Also, having proved mixed estimates for $C_8^{1,1,-1,1,1,-1,1,1}$, it is reasonable to think that the same method of proof works for operators that continue the sequence $1,1,-1,1,1,-1,...$ for arbitrarily long lengths. This is indeed the case as we prove in Theorem \ref{MT}, but some care has to be taken with the order in which we use suprema and Cauchy-Schwarz inequalities. At this point, we need to introduce a few definitions.
\begin{definition}
A set of consecutive positive integers $\{ i, ..., i +m\} \subset [n]$ forms a Lebesgue block $\mathfrak{B}$ for the operator $C_n^{\vec{\epsilon}}$ provided $\epsilon_{i+l} + \epsilon_{i+l+1} \not = 0$ for all $0 \leq l < m$. 
\end{definition}

\begin{definition}
We say a sign degeneracy occurs between indices $i$ and $i+1$ of the operator $C_n^{\vec{\epsilon}}$ if $\epsilon_i+ \epsilon_{i+1} =0$
\end{definition}
\begin{definition}
Suppose $C_n^{\vec{\epsilon}}: \otimes_{i=1}^n X^i \rightarrow L^q$ where for each $1 \leq i \leq n$, $X^i \in \{ L^{p_i},W_{p_i}\}.$ Then $W_{C_n^{\vec{\epsilon}}}^*:= \{ 1 \leq i \leq n : X^i = W_{p_i}.\}$
\end{definition}
\begin{theorem}[Main Theorem]\label{MT}
Fix $n \geq 2$, $\vec{\epsilon} \in \{\pm 1\}^n$. Form the operator $C_n^{\vec{\epsilon}}$ with domain $\otimes _{i=1}^n X^i$ and assume for every $i : 1 \leq i \leq n$ either $X^i = L^{p_i}$ for some $1 <p_i<\infty $ or $X^i = W_{p_i}$ for some $p_i >2$. Then $C_n^{\vec{\epsilon}}: \otimes _{i=1}^n X^i \rightarrow L^{\frac{1}{\sum_{i=1}^n \frac{1}{p_i}}}$ is bounded provided the following additional conditions hold:

1) The restricted maximal operator  is bounded on each Lebesgue block $\mathfrak{B}$.

2) If $\epsilon_{i-1}+\epsilon_i =0$ or $\epsilon_{i}+\epsilon_{i+1}=0$, then $\{i-1,i,i+1\} \cap W^*_{C_n^{\vec{\epsilon}}} \not = \emptyset$,

3) $\epsilon_i +\epsilon_{i+1} =0$ implies $\frac{1}{p_i} + \frac{1}{p_{i+1}} <1$.
\end{theorem}

\begin{proof}
 It is enough by the remark at the end of the previous theorem to assume all Lebesgue blocks have length 1 or 2. In fact, we now specialize to the case where each Lebesgue block has length 2 and there is only 1 function in a Wiener space between each Lebesgue block. The same type of proof will work for the cases where a Lebesgue block has length one or there is more than one Wiener function separating a Lebesgue block. So, we will restrict our attention to the operator $C_{3n-1}^{1,1,-1,1,1,-1,...,-1,1,1}$ and prove bounds for it. To this end, we introduce two martingale structures for each $f_{3i} \in W_{p_{3i}}$:

\begin{eqnarray*}
\gamma_{f_{3i}}(x)&=&\frac{ \int_{-\infty}^x |\hat{f}_{3i}(\bar{x})|^{p_{3i}^\prime}d\bar{x}}{|| \hat{f}_{3i}||_{p_{3i}^\prime}^{p_{3i}^\prime}}\\
_{3i}E^{m_1}_{j_1}&=&\gamma_{f_{3i}}^{-1}([j_12^{-m_1}, (j_1+1)2^{-m_1}))\\
\gamma_{m_1, j_1, f_{3i}}(x)&=& \frac{ \int_{-\infty}^x |\hat{f}_{3i}(\bar{x})|^{p_{3i}^\prime} \chi_{_{3i}E^{m_1}_{j_1}}(\bar{x})d\bar{x}}{ ||\hat{f}_{3i} \chi_{ _{3i}E^{m_1}_{j_1}} ||_{p_{3i}^\prime}^{p_{3i}^\prime}}\\
_{3i}E^{m_1, m_2}_{j_1, j_2} &=& \gamma_{m_1, j_1, f_{3i}}^{-1}([j_22^{-m_2}, (j_2 +1)2^{-m_2})).
\end{eqnarray*}
where we may assume as usual that each $\hat{f}_{3i} (x) \not =0~a.e.~x \in \mathbb{R}$ with appropriate modifications for the rightmost elements of the restricted martingale structures. Using the standard partition for each $i: 1 \leq i \leq n-1$

\begin{eqnarray*}
\mathbb{R}^3 \supset \{x_{3i-1} < x_{3i} < x_{3i+1}\}= \bigsqcup_{m_1, m_2} \bigsqcup_{j_1, j_2} ~_{3i}E^{m_1}_{j_1, l} \times (_{3i}E^{m_1}_{j_1,r} \cap~ _{3i}E^{m_1, m_2}_{j_1, j_2, l}) \times~ _{3i}E^{m_1, m_2}_{j_1, j_2, r},
\end{eqnarray*} 
we decompose $C_{3n-1}^{1,1,-1,1,1,-1,...,-1,1,1}$. By moving the sum over $2(n-1)$ scales outside the $L^q$ norm as usual, we are left inside with a sum over $2(n-1)$ indices $j_1, ..., j_{2(n-1)}$.  Of course, we can split the sum over $\vec{j}$ into $2^{n-1}$ pieces by restricting each  even index $j_{2k}$ either to $0 \leq j_{2k}<2^{m_{2k}}-1$ or to the endpoint $2^{m_{2k}}-1$.  We say a given even index $2k:1 \leq k \leq n-1$ is Type A if the corresponding $j_{2k}$ is restricted to $0 \leq j_{2k}<2^{m_{2k}}-1$ and $m_{2k}$ is Type B if $j_{2k}$ is restricted to $2^{m_{2k}}-1$. For example, in Theorem \ref{SMT} we broke apart the original sum into four smaller sums as follows:
\begin{eqnarray*}
\sum_{j_1, j_2, j_3, j_4} &=& \sum_{j_1,  j_2 \not =2^{m_2}-1, j_3, j_4 \not = 2^{m_4}-1} \\&+&\left.  \sum_{j_1, j_2 \not = 2^{m_2}-1, j_3} \right|_{j_4=2^{m_4}-1} \\&+& \left. \sum_{j_1, j_3,  j_4 \not = 2^{m_4}-1} \right|_{j_2=2^{m_2}-1} \\&+& \left. \sum_{j_1, j_3} \right|_{j_2=2^{m_2}-1, j_4=2^{m_4}-1}.
\end{eqnarray*}
For the first sum on the right hand side, the indices $2$ and $4$ are both Type A. For the second sum, the index $2$ is Type A and index $4$ is Type B. For the third sum, index $2$ is Type B and index $1$ is Type A. For the last sum, both $1$ and $2$ are Type B indices. For convenience, say that the first sum is type $AA$, the second type $AB$, the third type $BA$, and the fourth type $BB$.  
It is instructive to recall how sums of type AA in the decomposition of $C_8^{1,1,-1,1,1,-1,1,1}$ were handled. Setting
 \begin{eqnarray*}
A^{m_1}_{j_1}&=&BHT(f_1, f_2*\check{\chi}_{E^{m_1}_{j_1,l}})\\
B^{m_1, m_2}_{j_1, j_2}&=& f_3 *\check{\chi}_{E^{m_1}_{j_1, r}}  *\check{\chi}_{E^{m_1, m_2}_{j_1,j_2, l}}\\
C^{m_1, m_2, m_3}_{j_1, j_2, j_3}&=&BHT(f_4*\check{\chi}_{E^{m_1, m_2}_{j_1, j_2, r}}, f_5*\check{\chi}_{E^{m_3}_{j_3,l}})\\
D^{m_1, m_2}_{j_1, j_2} &=& f_6*\check{\chi}_{E^{m_3}_{j_3, r}} *\check{\chi}_{E^{m_3, m_4}_{j_3, j_4,l}}\\
F^{m_1, m_2}_{j_1, j_2}&=& BHT(f_7*\check{\chi}_{E^{m_3, m_4}_{j_3, j_4, r}}, f_8),
\end{eqnarray*}
we observed 
\begin{eqnarray*}
&&\left| \sum_{j_1, j_2 : j_2 \not = 2^{m_2}-1, j_3, j_4 \not = 2^{m_4}-1} A^{m_1}_{j_1} B^{m_1,m_2}_{j_1, j_2} C^{m_1, m_2, m_3}_{j_1, j_2, j_3} D^{m_3, m_4}_{j_3, j_4} F^{m_3, m_4}_{j_3, j_4} \right| \\&\leq&\sup_{j_1} | A^{m_1}_{j_1}| \left(  \sum_{j_1,  j_2 \not = 2^{m_2}-1} |B_{j_1, j_2}^{m_1, m_2}|^2 \right)^{1/2}  \sup_{j_3} \left( \sum_{j_1,  j_2\not = 2^{m_2}-1} |C^{m_1, m_2, m_3}_{j_1, j_2, j_3}|^2 \right)^{1/2}\times\\&& \left( \sum_{j_3, j_4:j_4 \not = 2^{m_4}-1} |D^{m_3,m_4}_{j_3, j_4}|^2 \right)^{1/2} \left( \sum_{j_3, j_4 :j_4 \not= 2^{m_4}-1} |F^{m_3,m_4}_{j_3,j_4}|^2 \right)^{1/2}.
\end{eqnarray*}

Note that we used a supremum and the Cauchy-Schwarz inequality for the pair $(j_1, j_2)$ \emph{before} proceeding to use a supremum and Cauchy-Schwarz inequality for the pair $(j_3, j_4)$. This order ensures that one takes the $l^2$ norm over $j_1, j_2$ for the cross factor $C_{m_1, m_2, m_3}$ before the $l^\infty$ norm over $j_3$. That the supremum over $j_3$ appears outside the sum over $j_1$ and $j_2$  is necessary for applying Lemma \ref{L2}. We summarize this observation with the heuristic that one resolves sums of type AA from ``left to right."

One quickly checks that for sums of type $AB$, the cross factor is$\left( \sum_{j_1, j_2:j_2 \not = 2^{m_2}-1, j_3} |C_{m_1, m_2, m_3}|^2 \right)^{1/2}$, which we were able to handle using standard $l^2$ vector-valued inequalities, while sums of type $BA$ gave us cross factors like $\sup_{j_1, j_3} |C_{j_1, 2^{m_2}-1, j_3}|$, which we could pass to the Bi-Carleson estimate. Neither sums of type $AB$ nor sums of type $BA$ required one to estimate the factors containing $(j_1, j_2)$ before those containing $(j_3, j_4)$ or the factors containing $(j_3, j_4)$ before those containing $(j_1, j_2)$. Lastly, sums of type BB required us to resolve from ``right to left." The cross factor took the form $\sup_{j_1} \left( \sum_{j_3} |C_{j_1, 2^{{m_2}-1}, j_3}|^2 \right)^{1/2}$. It is easy to check that resolving ``left to right" gives a convergent sum for a block consisting of an arbitrary number of As, and similarly resolving ``right to left" gives a convergent sum for a block consisting of an arbitrary number of Bs.  

Now, for a given sum in the decomposition of $\sum_{j_1, ..., j_{2(n-1)}}$, its type can be represented as a string of As and Bs of length $n-1$. This string can be separated into blocks of As and blocks of Bs of varying lengths. For each block of As, one resolves each j pair from ``left to right." Then, for each block of Bs, one resolves each j pair from ``right to left." Doing this yields a convergent geometric series for each of the $2^{(n-1)}$ pieces of the sum $\sum_{j_1, ...,j_{2(n-1)}}$.
\end{proof}

\appendix

\section{Two Lemmas using Khintchine's Inequality}
The first lemma is well known, see e.g. \citep{2030573}. The second lemma is seemingly new. 

 \begin{lemma}\label{L1}
Fix $\sigma$-finite measure spaces $X$ and $Y$. Let $0 <q \leq p <\infty$ and suppose $T:L^p(X) \rightarrow L^q(Y)$ is a continuous linear operator. Then for any sequence $\{f_j\}$ of functions in $L^p(X)$,
 
\begin{eqnarray*}
\left| \left|  \left( \sum_{j \in \mathbb{N}} | T(f_j)|^2 \right)^{1/2} \right| \right|_{L^q(Y)} \lesssim_{p,q} ||T||_{p \rightarrow q}\left| \left| \left( \sum_{j \in \mathbb{N}} |f_j|^2  \right)^{1/2} \right| \right|_{L^p(X)}.
\end{eqnarray*}
 
 \end{lemma}
 \begin{proof}
The standard proof linearizes using Khintchine as follows:

\begin{eqnarray*}
\left| \left|  \left( \sum_{j \in \mathbb{N}} | T(f_j)|^2 \right)^{1/2} \right| \right|_{L^q(Y)}  &\lesssim& \left| \left| \left( \mathbb{E} \left| \sum_{j \in \mathbb{N}} r_j(t)T( f_j) \right|^q \right)^{1/q} ~ \right| \right|_q \\ &=&  \left(  \mathbb{E}  \int_Y \left| T(\sum_{j \in \mathbb{N}} r_j(t) f_j ) \right|^q dx \right)^{1/q} \\&\leq&||T||_{p \rightarrow q} \left( \mathbb{E}  \left( \int_X  \left| \sum_{j \in \mathbb{N}} r_j(t) f_j\right| ^p  dx \right)^{q/p} \right)^{1/q} \\ &=& ||T||_{p \rightarrow q} \left( \int_X  \mathbb{E} \left| \sum_{j \in \mathbb{N}} r_j(t) f_j\right|^p dx \right)^{1/p} \\ &\lesssim& ||T||_{p \rightarrow q}\left( \int_X \left( \sum_{j \in \mathbb{N}} |f_j|^2 \right)^{p/2} dx \right)^{1/p}.
\end{eqnarray*}

\end{proof}

\begin{lemma}\label{L2}
Let $\{f^1_{j_1}\}$ and $\{f^n_{j_n}\}$ be any two sequences of functions in $L^{p_1}(\mathbb{R})$ and $L^{p_n}(\mathbb{R})$ respectively . Moreover, let $\sup C_n^{1,...,1} : \prod_{i=1}^n L^{p_i}(\mathbb{R}) \rightarrow L^{\frac{1}{\sum_{i=1}^n \frac{1}{p_i}}}$ be bounded with $1 < p_i <\infty~\forall 1 \leq i \leq n$.  Then, setting $q=\frac{1}{\sum_{i=1}^n \frac{1}{p_i}}$, 

\begin{eqnarray*}
  \left| \left| \sup_{I \subset \mathbb{R}}  \left(  \sum_{j_1, j_n} \left| C_n^{1,...,1}(f_{j_1}^1*\check{\chi}_{I}, f_2, ..., f^n_{j_n})\right|^2 \right)^{1/2} \right| \right|_q \lesssim ||\sup C^{1,...,1}_n||_{ \vec{p} \rightarrow q} \left| \left| \left( \sum_{j_1} |f^1_{j_1}|^2 \right)^{1/2} \right| \right|_{p_1} ... \left| \left| \left( \sum_{j_n}| f^n_{j_n}|^2 \right)^{1/2} \right| \right|_{p_n}.
\end{eqnarray*}
\end{lemma}

\begin{proof}
The proof again linearizes using Khintchine:
\begin{eqnarray*}
&& \left| \left| \sup_{I \subset \mathbb{R}}  \left(  \sum_{j_1, j_n} \left| C_n^{1,...,1}(f_{j_1}^1*\check{\chi}_{I}, f_2, ..., f^n_{j_n})\right|^2 \right)^{1/2} \right| \right|_q  \\ &\lesssim&  \left| \left| \sup_{I \subset \mathbb{R}}  \left( \mathbb{E} \left| \sum_{j_1, j_n} C_n^{1,...,1}(r^1_{j_1}(t)f_{j_1}^1*\check{\chi}_{I}, f_2, ..., f^n_{j_n}r^2_{j_n}(t))\right| ^q\right)^{1/q} \right| \right|_q \\ &=&  \left( \int  \sup_{I \subset \mathbb{R}} \mathbb{E} \left| C_n^{1,...,1}(\sum_{j_1} r^1_{j_1}(t)f_{j_1}^1*\check{\chi}_{I}, f_2, ..., \sum_{j_n} f^n_{j_n}r^2_{j_n}(t))\right| ^q dx \right)^{1/q}\\&\leq&  \left(   \mathbb{E} \int \sup_{I \subset \mathbb{R}} \left| C_n^{1,...,1}(\sum_{j_1} r^1_{j_1}(t)f_{j_1}^1*\check{\chi}_{I}, f_2, ..., \sum_{j_n} f^n_{j_n}r^2_{j_n}(t))\right| ^q dx \right)^{1/q} \\ &\lesssim& ||\sup C^{1,...,1}_n||_{\vec{p} \rightarrow q} \left( \mathbb{E}_1 ||\sum_{j_1} r^1_{j_1}(t) f^1_{j_1} ||^{q}_{p_1} \mathbb{E}_2 ||\sum_{j_n} r^2_{j_n}(t) f^n_{j_n}||_{p_n}^q \right)^{1/q} \\ &\lesssim& ||\sup C^{1,...,1}_n||_{\vec{p} \rightarrow q} \left( (\mathbb{E}_1 ||\sum_{j_1} r^1_{j_1}(t) f^1_{j_1} ||^{p_1}_{p_1})^{q/p_1}( \mathbb{E}_2 ||\sum_{j_n} r^2_{j_n}(t) f^n_{j_n}||_{p_n}^{p_n} )^{q/p_n} \right)^{1/q}.
\end{eqnarray*}
Using Fubini and Khintchine again, we arrive at the upper bound 

\begin{eqnarray*}
 ||\sup C^{1,...,1}_n||_{ \vec{p} \rightarrow q} \left| \left| \left( \sum_{j_1} |f^1_{j_1}|^2 \right)^{1/2} \right| \right|_{p_1} ... \left| \left| \left( \sum_{j_n}| f^n_{j_n}|^2 \right)^{1/2} \right| \right|_{p_n}.
  \end{eqnarray*}
\end{proof}


%% References
%% Usage of \citep is as follows:
%%   \citep{key}          ==>>  [#]
%%   \citep[chap. 2]{key} ==>>  [#, chap. 2]
%%   \citept{key}         ==>>  Author [#]

\bibliographystyle{plainnat}
\bibliography{bibPAP.bib}






% \end{thebibliography}

\end{document}